\documentclass[12pt]{article}
\usepackage{amsmath, amsthm, amssymb,amscd, mathabx,wasysym,verbatim,hyperref,mathrsfs,authblk,cleveref, tikz,graphicx}
\usepackage[all]{xy}
\theoremstyle{plain}
\newtheorem{thm}{Theorem}[section]

\newtheorem{lem}[thm]{Lemma}
\newtheorem{cor}[thm]{Corollary}

\newtheorem{prop}[thm]{Proposition}

\theoremstyle{definition}
\newtheorem{defn}[thm]{Definition}

\theoremstyle{remark}
\newtheorem{rem}[thm]{Remark}

\newcommand{\nc}{\newcommand} 
\nc{\hb}{\mathbb} 
\nc{\M}{\mathcal} 
\nc{\mf}{\mathfrak}
\nc{\mbf}{\mathbf}
\nc{\DMO}{\DeclareMathOperator}

\newbox\noforkbox \newdimen\forklinewidth
\setbox0\hbox{$\textstyle\smile$}
\setbox1\hbox to \wd0{\hfil\vrule width \forklinewidth depth-2pt
 height 10pt \hfil}
\setbox\noforkbox\hbox{\lower 2pt\box1\lower 2pt\box0\relax}
\def\anchor{\mathop{\copy\noforkbox}\limits}
 
\setbox0\hbox{$\textstyle\smile$}
\setbox1\hbox to \wd0{\hfil{\sl /\/}\hfil}
\setbox2\hbox to \wd0{\hfil\vrule height 10pt depth -2pt width
               \forklinewidth\hfil}
\newbox\doesforkbox
\setbox\doesforkbox\hbox{\box1 \lower 2pt\box2\lower2pt\box0\relax}
\def\nanchor{\mathop{\copy\doesforkbox}\limits}
 
\nc{\cA}{{\M A}} \nc{\cB}{{\M B}} \nc{\cC}{{\M C}} \nc{\cD}{{\M D}}
\nc{\cE}{{\M E}} \nc{\cF}{{\M F}} \nc{\cG}{{\M G}} \nc{\cH}{{\M H}}
\nc{\cI}{{\M I}} \nc{\cJ}{{\M J}} \nc{\cK}{{\M K}} \nc{\cL}{{\M L}}
\nc{\cM}{{\M M}} \nc{\cN}{{\M N}} \nc{\cO}{{\M O}} \nc{\cP}{{\M P}}
\nc{\cQ}{{\M Q}} \nc{\cR}{{\M R}} \nc{\cS}{{\M S}} \nc{\cT}{{\M T}}
\nc{\cU}{{\M U}} \nc{\cV}{{\M V}} \nc{\cW}{{\M W}} \nc{\cX}{{\M X}}
\nc{\cY}{{\M Y}} \nc{\cZ}{{\M Z}}
\nc{\Aa}{{\hb A}} \nc{\Cc}{{\hb C}} \nc{\Gg}{{\hb G}}
\nc{\Nn}{{\hb N}} \nc{\Pp}{{\hb P}} 
\nc{\Qq}{{\hb Q}} \nc{\Rr}{{\hb R}} \nc{\Zz}{{\hb Z}}
\nc{\mfa}{{\mf a}} \nc{\mfb}{{\mf b}} \nc{\mfk}{{\mf k}}
\nc{\mfm}{{\mf m}} \nc{\mfp}{{\mf p}} \nc{\mfq}{{\mf q}}
\nc{\mfr}{{\mf r}}
\nc{\fP}{{\mf P}}
\DMO*{\trdeg}{td}
\DMO*{\spec}{Spec}
\DMO*{\fork}{\nanchor}
\DMO*{\dnf}{\anchor}
\DMO{\RU}{RU}
\DMO{\deter}{det}
\DMO{\RM}{RM}
\DMO{\RC}{RC}
\DMO{\Real}{Re}
\DMO{\Imag}{Im}
\DMO{\tr}{tr}
\DMO{\qc}{QC}
\DMO{\Hu}{Hull}
\DMO{\leg}{length}
\DMO{\area}{area}
\DMO{\dia}{diameter}
\DMO{\iso}{Iso}
\DMO{\dis}{dist}
\DMO{\grad}{grad}
\DMO{\vol}{volume}
\DMO{\gra}{grad}
\DMO{\hd}{nbhd}
\DMO{\dv}{div}
\DMO{\Psl}{PSL}
\nc{\Mb}{\mathfrak^{2b/\delta}_{K_x}}
\nc{\Ma}{\mathfrak^{2a/\delta}_{K_x}}
\nc{\dif}{\mathrm{d}}
\nc{\G}{\Gamma}
\nc{\g}{\gamma}
\nc{\D}{\nabla}
\nc{\p}{\partial}
\nc{\DD}{\Delta^2}
\nc{\pp}{\partial^2} 
\nc{\de}{\delta}
\nc{\td}[2]{\trdeg{({#1}/{#2})}}
\nc{\dtd}[2]{\trdeg_{\delta}{({#1}/{#2})}}
\nc{\dspec}[1]{\spec_{\delta}{#1}}
\nc{\ddim}[1]{\dimen_{\delta}{#1}}
\nc{\gens}[1]{\langle {#1} \rangle}        
\nc{\gen}[2]{ {#1} \langle {#2} \rangle } 
\nc{\form}{\Omega}
\nc{\set}[1]{\left\{ {#1} \right\}}
\nc{\mr}{\hat}
\nc{\pr}{\partial}
\nc{\bc}[3]{\cB^{#1}({#2},{#3})=B^{#1}_{#2}(C^{#2}_{#3}(t)+B^{#1}_{#3}(C^{#2}_{#3}(t))} 
\nc{\tuple}[2]{{#1},\ldots,{#2}} \nc{\ptu}[2]{{#1}:\ldots:{#2}}

\nc{\maps}[3]{{#1}\!:\!{#2}\rightarrow{#3}}
\nc{\map}[2]{{#1}\rightarrow {#2}} \nc{\res}[2]{{#1} |_{#2}}
\nc{\imbed}{\hookrightarrow}
\title{The classification of Kleinian groups of Hausdorff dimensions at most one and Burnside's conjecture}
\author{Yong Hou\footnote{Primary:$57\mbox{M}50; 53\mbox{C}30$. Secondary:$22\mbox{A}05.$ Supported by Ambrose Monell Fundation.} }
\date{}
\begin{document}
\maketitle
\begin{abstract}
In this paper we provide the complete classification of convex cocompact Kleinian group of Hausdorff dimensions less than $1.$ In particular, we prove that every convex cocompact Kleinian group of Hausdorff dimension $<1$ is a classical Schottky group. This upper bound is sharp. The result implies that the converse of Burside's conjecture \cite{Burside} is true: All non-classical Schottky groups must have Hausdorff dimension $\ge1$. The upper bounds of Hausdorff dimensions of classical Schottky groups
has long been established by Phillips-Sarnak \cite{phillips} and Doyle \cite{Doyle}. The prove of the theorem relies on the result of Hou \cite{Hou}.\end{abstract}
\setcounter{tocdepth}{1}
\section{Introduction and Main Theorem}
We take Kleinian groups to be finitely generated, discrete subgroups of $\text{PSL}(2,\mathbb{C}).$ The main theorem is:
\begin{thm}[Classification]\label{main}
All convex cocompact Kleinian groups $\G$ with limit set of Hausdorff dimension $<1$ are classical Schottky groups. This bound is sharp.
\end{thm}
It was pointed out by Sarnak to me that, originally Burnside conjectured that all classical Schottky groups must have Hausdorff dimension at most one. However Burnside's conjecture was disapproved by Myrberg \cite{my, Doyle}. Theorem \ref{main}, implies the converse of Burnside's conjecture is in fact true.
\begin{cor}[Converse of Burnside's conjecture \cite{Burside}]
All non-classical Schottky groups must have Hausdorff dimension $\ge1.$ \end{cor}
Finally we note that the proof of Theorem \ref{main} relies on the result of \cite{Hou} Theorem \ref{HouY}.  The proof of \ref{HouY} is completely different and of independent interests from the current works. 

\begin{thm}[Hou\cite{Hou}]\label{HouY}
There exists $\lambda>0$ such that any Kleinian group with limit set of Hausdorff dimension $<\lambda$ is a classical Schottky group.
\end{thm}
Finally, we note that Theorem \ref{main} completes the picture of  lower spectrum of the Schottky groups dimensions which compliments the upper bounds established in \cite{phillips, Doyle}. 
\subsection{Strategy of proof}
Denote $\mathfrak{J}_g$ to be the rank-$g$ Schottky space. Denote $\mathfrak{J}_{g,o}$ the open subset of classical Schottky groups of $\mathfrak{J}_g.$ For
$\G\in\mathfrak{J}_g$ denote by $\mathfrak{D}_\G$ the Hausdorff dimension of $\Lambda_\G$, the limit set of $\G.$ Let $\mathfrak{J}^1_g$ denote the open subset of $\mathfrak{J}_g$ consists of $\G$ Schottky groups with Hausdorff dimension $\mathfrak{D}_\G<1.$ The main claim of our proof can then be stated as:
\[\mathfrak{J}_{g,o}\cap\mathfrak{J}_g^1=\mathfrak{J}^1_g .\]
Since $\mathfrak{J}_{g,o}\cap\mathfrak{J}_g^1\subset\mathfrak{J}^1_g$ is open subset of $\mathfrak{J}^1_g$, our proof essentially consists of two main parts as follows:
\begin{itemize}
\item[(1)]
$\mathfrak{J}_{g,o}\cap\mathfrak{J}_g^1$ is also closed in $\mathfrak{J}^1_g$.
\item[(2)]
Every connected components of $\mathfrak{J}^1_g$ contains a point in $\mathfrak{J}_{g,o}\cap\mathfrak{J}_g^1.$
\end{itemize}
Note that it follows from part $(1)$ that we have $\mathfrak{J}_{g,o}\cap\mathfrak{J}_g^1$ consists of connected components of $\mathfrak{J}_g^1.$
\par
The proof of part $(1)$ is done as follows. It is a result of Bowen \cite{Bowen} that, a Schottky group $\G$ has Hausdorff dimension $<1,$ if and only if there exist a \emph{rectifiable} $\G$-invariant closed curve. Let $\mathscr{R}(S^1,W)$ be the space of bounded length closed curves which intersects the compact set $W\subset\mathbb{C}$ and equipped with Fr\'echet metric. It is complete space, see Section $2$. 
We show that if Hausdorff dimension $<1$ then, every quasi-circle with bounded length of $\G$ is the limit of a sequence of quasi-circles of $\G_n$ in $\mathscr{R}(S^1,W).$ We also show that 
if $\G$ is a Schottky group, then every quasi-circle of $\G$ has an open neighborhood in the relative topology of $\Psi_\G$ (see section 3) such that, every element of the open neighborhood is a quasi-circle  of $\G.$ We also define linearity and transversality invariant for quasi-circles, and show that quasi-circles of classical Schottky groups preserve these invariants, and non-classical Schottky groups do not have transverse linear quasi-circles.\par
Given a quasi-circle of a Schottky group $\G$, we show that there exists an open neighborhood (in the relative topology of space of rectifiable curves with respect to Frechet metric) about the quasi-circle such that, every point in the open neighborhood is a quasi-circle of $\G,$ see Lemma \ref{open}. Next assume that we have a sequence of
classical Schottky groups $\G_n\to\G$ to a Schottky group, and are all of Hasudorff dimensions less than one. 
We then study singularity formations of classical fundamental domains of $\G_n$ when $\G_n\to\G.$ These singularities are of three types: tangent, degenerate,
and collapsing. We show that all these singularities will imply that there exists a quasi-circle such that, every open neighborhood about this quasi-circle will contain some points which is \emph{not} a quasi-circle. Essentially, the existence of a singularity will be \emph{obstruction} to the existence of any open neighborhood that are of quasi-circles, 
see Lemma \ref{singular}. Hence it follows from these results that, if $\G_n\to\G$ with $\G_n$ classical and, all Hausdorff dimensions are of less than one then $\G$ must be a classical Schottky group.\par
The proof of part $(2)$ follows from our stronger result which is Lemma \ref{path-connect}. Where we show that one can connect any Schottky group 
$\G\in\mathfrak{J}_g^1$ by a continuous path in $\{\G_t\}_{t\in [0,1]}\in\mathfrak{J}_g^1$ to some classical Schottky group with nonincreasing Hausdorff dimension. The proof of this part relies on Theorem \ref{HouY} in \cite{Hou} and Hausdorff dimension $<1$ of $\G$.
\par

\section{Quasi-circles and generating Jordan curves }
Schottky group $\G$ of rank $g$  is defined as convex-cocompact discrete faithful representation of the free group $\mathbb{F}_g$ in $\text{PSL}(2,\mathbb{C}).$ It follows that 
$\G$ is freely generated by purely loxodromic elements $\{\g_i\}_1^g$. This implies we can find collection of 
open topological disks $ D_{i}, D_{i+g}, 1\le i\le g$ of disjoint closure $\bar D_i\cap\bar D_{i+g}=\emptyset$ in the Riemann sphere $\partial\mathbb{H}^3=\overline{\mathbb{C}}$ with
boundary curves $\partial\bar D_{i}=c_{i},\partial \bar D_{i+g}=c_{i+g}.$ By definition $c_{i},c_{i+g}$ are closed Jordan curves in Riemann sphere 
$\partial\mathbb{H}^3,$ such that $\g_i(c_{i})=c_{i+g}$ and $\g_i(D^o_{i})\cap D^o_{i+g}=\emptyset.$ Whenever there exists a
set $\{\g_1,...,\g_g\}$ of generators with all $\{c_{i},c_{i+g}\}_1^g$ as circles, then it is 
called a classical Schottky group with $\{\g_1,...,\g_g\}$ classical generators.\par
Schottky space $\mathfrak{J}_g$ is defined as space of all rank $g$ Schottky groups up to conjugacy by $\text{PSL}(2,\mathbb{C}).$ By normalization,
we can chart $\mathfrak{J}_g$ by $3g-3$ complex parameters. Hence $\mathfrak{J}_g$ is $3g-3$ dimensional complex manifold. The bihomolomorphic 
$\text{Auto}(\mathfrak{J}_g)$ group is $\text{Out}(\mathbb{F}_g),$ which is isomorphic to quotient of the handle-body group. Denote by $\mathfrak{J}_{g,o}$ the set of all elements of $\mathfrak{J}_g$ that are classical Schottky groups. Note that $\mathfrak{J}_{g,o}$ is open in $\mathfrak{J}_g.$ On the other hand it is nontrivial result due to Marden that $\mathfrak{J}_{g,o}$ is non-dense subset of $\mathfrak{J}_g.$ However, it follows from
Theorem\ref{HouY}, subset of $\mathfrak{J}_g$ with Hausdorff dimension less than some $\lambda$ is in $\mathfrak{J}_{g,o}$ and is $3g-3$ dimensional open connected submanifold consists of classical Schottky . \\
\\
\noindent{\bf  Notations:} 
\begin{itemize}
\item
Given $\G$ a Kleinian group, we denote by $\Lambda_\G$ and $\Omega_\G$ and $\mathfrak{D}_\G$ its limit set, region of discontinuity,
and Hausdorff dimension respectively throughout this paper. 
\item
$\mathfrak{J}_g^\lambda=\{\G\in\mathfrak{J}_{g} | \mathfrak{D}_\G<\lambda\}$ and $\mathfrak{J}^\lambda_{g,o}=\{\G\in\mathfrak{J}_{g,o} | \mathfrak{D}_\G<\lambda\}$, for  $\lambda>0.$
\item
$\overline{\mathfrak{J}^\lambda_{g,o}}$ defined as: the closure of set of all classical Schottky groups with Hausdorff dimension $\le\lambda$ in $\mathfrak{J}_g.$

\item
Given a fundamental domain $\mathcal{F}$ of $\G$, we denote the orbit of $\mathcal{F}$ under actions of $\G$ by $\mathcal{F}_\G.$ We also say $\mathcal{F}_\G$
is a classical fundamental domain of classical Schottky group if $\partial\mathcal{F}$ are disjoint circles.
\end{itemize}
\begin{defn}[Quasi-circles]
Given a geometrically finite Kleinian group $\G,$ a closed $\G$-invariant Jordan curve that contains the limit set  $\Lambda_\G$ is called \emph{quasi-circle} of $\G$.
\end{defn}

Next we give a construction of quasi-circles of $\G$ which is a generalization of the construction by Bowen \cite{Bowen}. \par
Let $\mathcal{F}$ be a fundamental domain of $\G$, and $\{c_i\}_1^{2g}$ be the collection of $2g$ disjoint Jordan curves comprising $\partial\mathcal{F}.$  Let $\zeta$ denote collection of arcs $\zeta=\{\zeta_i\}$ connecting points $p_i\in c_i,p'_i\in c_{i+g}$  for $1\le i\le g,$ and 
arcs on $c_{i+g}$ that connects $p_{i+g}$ to $\g_i(p_i)$ and $p'_{i+g}\in c_{i+g}$ to $\g_i(p'_i).$
So $\zeta$ is a set of $g$ disjoint curves connecting disjoint points on collection of Jordan curves of $\partial\mathcal{F}$ (Figure 1 ). \par
\begin{figure}[ht!]
\centering
\includegraphics[width=45mm]{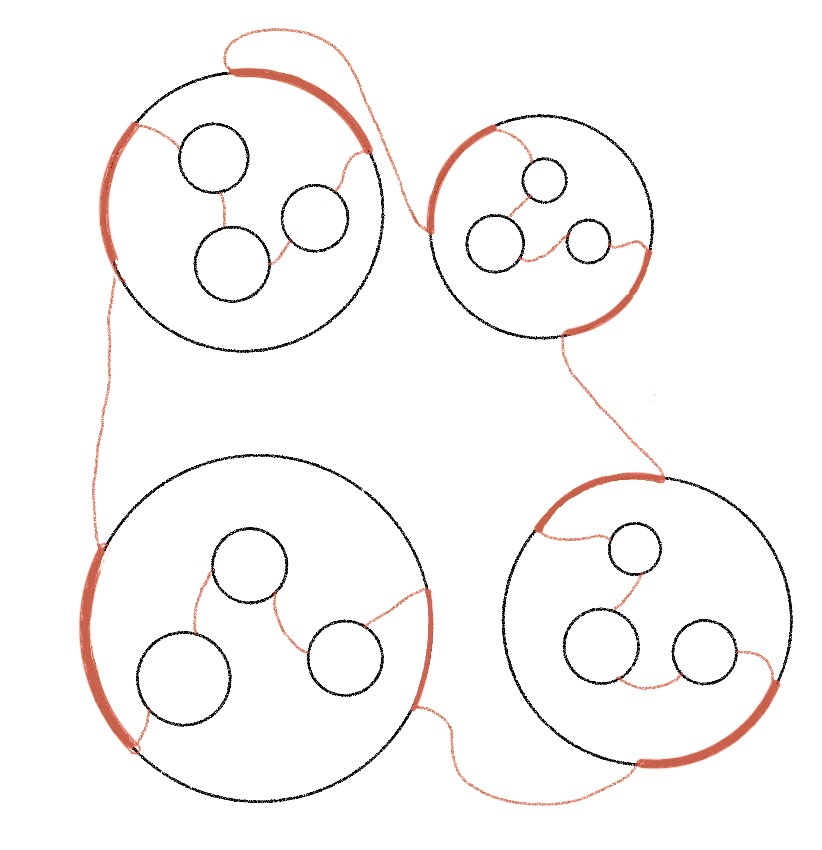}  
\caption{Quasi-circle \label{overflow}}
\end{figure}
$\eta_\G=\Lambda_\G\cup\cup_{\g\in\G}\g(\zeta)$ defines a $\G$-invariant closed curve containing $\Lambda_\G.$ 
$\eta_\G$ defines a quasi-circle of $\G$. Obviously there are infinitely many quasi-circles and different $\zeta$ gives a different quasi-circles. Note that, the simply connected regions $\mathbb{C}\setminus\eta_\G$ gives the Bers simultaneous uniformization of Riemann surface $\Omega_\G/\G.$

\begin{defn}[Generating curve]\label{marking}
Given a quasi-circle $\eta_\G$ of $\G.$ We say a collection of disjoint curves $\zeta$ is a generating curve of $\eta_\G$, if $\eta_\G$ can be generated by $\zeta.$
\end{defn}

Note that the quasi-circles constructed in \cite{Bowen}, which requires that $p_{i+g}$ is a imagine of $p_i$ under element of $\g_i$, is a subset of the collection that we have defined here. In fact, this generalization is also used for the construction of quasi-circles of non-classical Schottky groups. 

\begin{prop}\label{marking}
Every quasi-circle $\eta_\G$ of $\G$ is generated by some generating curves $\zeta.$
\end{prop}
\begin{proof}
Let $\eta_\G$ be a quasi-circle of $\G.$ Let $\mathcal{F}$ be a fundamental domain of $\G.$ Set $\xi=\mathcal{F}\cap\eta_\G$ and 
$\bar\xi=\partial\mathcal{F}\cap\eta_\G.$ Then $\xi,\bar\xi$ consists of collection of disjoint curves which only intersects along $\partial\mathcal{F}.$
Hence we have, $\g(\xi)\cap\g'(\xi)=\emptyset$ for $\g,\g'\in\G$ with $\g\not=\g'.$ Since $\overline{\G(\xi)}-\G(\xi)=\Lambda_\G$ we have have 
$\eta_\G=\Lambda_\G\cup\G(\xi\cup\bar\xi),$ hence $\xi\cup\bar\xi$ is a generating curve of $\eta_\G.$

\end{proof}

\begin{defn}[Linear quasi-circle]
We call a quasi-circle $\eta_\G$ \emph{linear} if, 
 $\eta_\G\backslash\Lambda_\G$ consists of points, circular arcs or  lines.
\end{defn}

Note that, if $\eta_\G$ is linear then there exists  $\mathcal{F}_\G$ such that, $\eta_\G\cap \mathcal{F}_\G$ and $\eta_\G\cap \partial\mathcal{F}_\G$ 
are piece-wise circular arcs or lines.\par
We say an arc $\zeta\subset\eta_\G\cap \mathcal{F}_\G$
is orthogonal if the tangents at intersections on $\partial\mathcal{F}_\G$
are orthogonal with $\partial\mathcal{F}_\G,$ and an arc $\xi\subset\eta_\G\cap \partial\mathcal{F}_\G$ is parallel if $\xi\subset\partial\mathcal{F}_\G.$

\begin{defn}[Right-angled quasi-circle]
Given a linear quasi-circle $\eta_\G$ of $\G$, if all linear arcs intersect at right-angle then we say $\eta_\G$ is \emph{right-angled quasi-circle}.
\end{defn}

\begin{defn}[Transverse quasi-circle]
Given a quasi-circle $\eta_\G$ of $\G$, we say $\eta_\G$ is \emph{transverse} quasi-circle if $\eta_\G$ intersects $\partial\mathcal{F}_\G$ orthogonally for some 
$\mathcal{F}_\G$ and, $\eta_\G$ have no parallel arc. Otherwise, we say $\eta_\G$ is non-transverse.

\end{defn}

\begin{defn}[Parallel quasi-circle]
Given a quasi-circle $\eta_\G$ of $\G$, we say $\eta_\G$ is \emph{parallel} quasi-circle if there exists some arc $\eta$ of $\eta_\G$ such that 
$\eta\subset\partial\mathcal{F}_\G$

\end{defn}

\begin{prop}
Transverse quasi-circles always exists for a given Schottky group $\G.$ 
\end{prop}
\begin{proof}
Let $\mathcal{F}$ be bounded by $2g$ distinct Jordan closed curves and take any curve connecting $p_i\in c_i,p_{i+g}\in c_{i+g}$ such that $p_{i+g}=\g_i(p_i)$ and $p_{i+g}\in c_{i+g}=\g_i(p_{i+g})$ for $1\le i\le g,$ that intersects $c_i,c_{i+g}$
orthogonally.
\end{proof}

It should be noted that a quasi-circle of $\G$ in general is not necessarily \emph{rectifiable}. For instance, if we take $\zeta$ to be some non-rectifiable generating curves then,
$\eta_\G$ will be non-rectifiable. Recall a curve is said to be rectifiable if and only if the $1$-dimensional Hausdorff measure of the curve is finite.
This is not the only obstruction to rectifiability, in fact we have the following result of Bowen:

\begin{thm}[\cite{Bowen}]\label{Bowen}
For a given Schottky group $\G,$ the Hausdorff dimension of limit set is $\mathfrak{D}_\G<1,$ if and only if there exists a rectifiable quasi-circle for $\G.$
\end{thm}
The proof of Theorem \ref{Bowen} relies on the fact that the Poincare series of $\G$ converges if and only if $\mathfrak{D}_\G<1$ \cite{Bowen}.
\begin{prop}\label{rec}
Let $\G$ be a Schottky group of $\mathfrak{D}_\G<1.$ Suppose a given generating curve $\zeta$ is a rectifiable curve. Then $\eta_\G$ is a rectifiable quasi-circle of $\G.$
\end{prop}
\begin{proof}
Let $\mu^1$ be the $1$-dimensional Hausdorff measure. Since $\mathfrak{D}_\G<1,$ we have $\mu^1(\Lambda_\G)=0.$ Let $\G=\{\g_k\}_{k=1}^\infty$, 
let $\g'_k$ denotes the derivative of $\g_k.$ Then we have,   
\begin{align*}
\mu^1(\eta_\G)&=\sum^\infty_{k=1}\mu^1(\g_k(\zeta)) \\
&\asymp\mu^1(\zeta)\sum^\infty_{k=1}|\g'_k(z)|,  \quad z\in \zeta.
\end{align*} 
Also since $\mathfrak{D}_\G<1$ if and only if Poincare series satisfies $\sum^\infty_{k=1}|\g'_k|<\infty.$ This implies that $\eta_\G$ is rectifiable if and only if 
$\zeta$ rectifiable.       
\end{proof}
Let $W\subset\mathbb{C}$ be a compact set. Denote the space of closed curves with \emph{bounded} length in $\mathbb{C}$ that intersect with $W$ by: 
\[\mathscr{R}(S^1,W)\subset\{h:S^1\to\mathbb{C}| h(S^1)\cap W\not=\emptyset,\quad\text{$h$ continuous rectifiable map}\}.\]
For $h_1,h_2\in\mathscr{R}(S^1,W)$ let $\ell(h_1),\ell(h_2)$ be it's respective arclength. The Fr\'{e}chet distance is defined as,
 \[d_F(h_1,h_2)=\inf\{\sup|h_1(\sigma_1)-h_2(\sigma_2)|;\sigma_1,\sigma_2\in\text{Homeo}(S^1)\}+|\ell(h_1)-\ell(h_2)|.\]
For a given compact $W\subset\mathbb{C},$ the space of closed curves with bounded length $\mathscr{R}(S^1,W)$ is a metric space with respect to $d_F.$ 
Two curves in $\phi,\psi\in\mathscr{R}(S^1,W)$ are same if there exists parametrization $\sigma$ such that $\psi(\sigma)=\phi$ and  $\ell(\phi)=\ell(\psi).$
The topology on $\mathscr{R}(S^1,W)$ is defined with respect to the metric $d_F$, see \cite{BR} p388. Let $\xi$ be a generating curve for $\eta_\G$. Fix a indexing of $\G$, set $\xi_i=\cup^i_1\g_j(\xi)$. Let $\sigma_i$ be a parametrization of $\xi_i$ such that 
$\sigma=\cup_i\sigma_i$ define a parametrization of $\eta_\G.$ Then $d_F(\eta^1_\G,\eta^2_\G)\le M(\inf_{\sigma^1,\sigma^2}|\xi(\sigma^1_k)-\zeta(\sigma_k^2)|+|\ell(\xi)-\ell(\zeta)|)$ for some $M, k.$ This implies continuity of $\eta_\G$ with respect to generating curve $\xi.$

\begin{prop}\label{curves-space}
For a given compact $W\subset\mathbb{C},$ the space $\mathscr{R}(S^1,W)$ is complete metric space with respect to $d_F.$
 \end{prop}
 \begin{proof}
Let $\{\phi_i\}\subset\mathscr{R}(S^1,W)$ be a Cauchy sequence. $\phi_i$ are rectifiable curves of bounded length and so there exists $\{\sigma_i\}$ Lipschitz parameterizations with bounded Lipschitz constants, such that
$\{\phi_i(\sigma_i(t))\}$ are uniformly Lipschitz. Then completeness follows from
the fact that all curves of $\mathscr{R}(S^1,W)$ are contained within some large compact subset of $\mathbb{C}.$
 \end{proof} 

The connectivity of $\mathfrak{J}_g^{1}$ was not known previously. In fact, very little is known on the topological and geometrical structure of $\mathfrak{J}^1_g$ in general. In this direction we prove our next proposition.

 \begin{prop}\label{H}
Every connected component of $\mathfrak{J}_g^{1}$ contains a point in $\mathfrak{J}_{g,o}\cap\mathfrak{J}_g^{1}.$
\end{prop}
The proof of Proposition \ref{H} will follow from Lemma \ref{path-connect}, a much stronger statement. We shall construct and shows the existence of a path in $\mathfrak{J}^1_g$ such that it is non-increasing in Hausdorff dimensions.

\begin{lem}\label{path-connect} Let $\G$ a Schottky group with $\mathfrak{D}_{\G}<1.$
There exists a path $\{\G_{t}\}\subset\mathfrak{J}^1_g, t\in[0,1]$ in $\mathfrak{J}^1_g$ such that,
$\G_1$ is classical Schottky group. 
\end{lem}

The construction of $\{\G_{t}\}$ relies on the condition of $\mathfrak{D}_{\G}<1$ and the use of Theorem \ref{HouY} as we show in the following.
\begin{proof}
Let $B(\G)$ be the space of unit ball Beltrami differentials of $\G$. For $\mu\in B(\G)$ and we set $\mu^{\epsilon_1}=\epsilon_1\mu$ for small $\epsilon_1>0$, and 
let $f_{\epsilon_1}$ be the corresponding  quasi-conformal map. We will next construct new quasi-conformal $\tilde{f}_{\epsilon_1}$ based on 
$f_{\epsilon_1}.$ Note that we take $f_0$ to be the identity.
\par
Let $\mathcal{F}_\G$ Schottky domain bounded by disjoint closed Jordan curves $\{\mathcal{C}_i\}_1^{2g}$ generated by $\{\g_i\}_1^{2g}.$ Denote by $\mathcal{D}_i$ the closed topological disk bounded by $\mathcal{C}_i=\partial\mathcal{D}_i.$ Note that we also have some small disjoint open neighborhoods $\mathcal{D}_i\subset U_i.$ 
Let $\mathcal{D}_{i_j}$ be the topological disk bounded by image of $\mathcal{C}_i$ under admissible string of $\g_j$, and similarly $U_{i_j}$ denotes images of $U_i$. Here we say string $\g_{i_1}...\g_{i_l}...\g_{i_j}$ is admissible for $\mathcal{D}_i$ if $i_j\not=i$, and $|i_l-i_{l+1}|\not=g$. 
Set $V_{i_j}=U_{i_j}\backslash\mathcal{D}_{i_j}$. Then the limit set $\Lambda_\G$ is given by 
$\cap_{k=1}^\infty\cup_{|j|=k}\mathcal{D}_{i_j}$ for admissible string $i_j\in\{1,...,2g\}^k.$
We define sequence of quasiconformal maps $\{f_{\epsilon_1,k}\}$ as follows. 
Set $f_{\epsilon_1,k}=f_{\epsilon_1}$ on $\mathbb{C}\backslash\cup_{|j|=k}U_{i_j}$. Let $\partial\mathcal{D}^\tau_{i_j}$ be continuous family of Jordan curves which are shrinks of $\partial\mathcal{D}_{i_j}$ to a point into the interior of $\mathcal{D}_{i_j}$ as $\tau\to\infty.$  Let $\phi^\tau_{i_j}$ be conformal maps $\mathcal{D}_{i_j}\rightarrow\mathcal{D}^\tau_{i_j}$. We define $f_{\epsilon_1,k}=\phi^{\epsilon_1}_{i_j}$ on 
$\mathcal{D}_{i_j}$ for 
all $i_j\in\{I,...,2g\}^k.$ Let $\psi^{\epsilon_1}_{i_j}$ be quasiconformal maps on $V_{i_j}$ which connects $f_{\epsilon_1}$ on $\mathbb{C}\backslash\cup_{|j|=k}U_{i_j}$ to $\phi^{\epsilon_1}_{i_j}$ on $\mathcal{D}_{i_j}.$ Note that the existence of $\psi^{\epsilon_1}_{i_j}$ follows from 
the fact that $f_{\epsilon}|_{\partial U_{i_j}}$ and $\phi^{\epsilon_1}_{i_j}|_{\partial\mathcal{D}_{i_j}}$ are quasisymmetric maps, hence we have
quasiconformal extension to annulis $V_{i_j}.$
Finally we set $f_{\epsilon_1,k}=\psi^{\epsilon_1}_{i_j}$ on $V_{i_j}$ for all ${i_j}.$\par
The sequence of quaisconformal maps satisfies the consistence condition $f_{\epsilon_1,k-1}=f_{\epsilon_1,k}$ on $\mathbb{C}\backslash\cup_{|j|=k}U_{i_j}$ where $f_{\epsilon_1,0}$ is defined to be $f_{\epsilon_1}.$ By our construction we have for some $M>0$ such that all dilations of the sequence $f_{\epsilon_1,k}$ are bounded $K_k<M$ i.e. $\|\mu^{\epsilon_1}_k\|_\infty<\rho$ for some $\rho<1$ and all $k$. Now since $\Lambda_\G=\cap_{k=1}^\infty\cup_{|j|=k}\mathcal{D}_{i_j}$ which is of Hausdorff dimension $<1$, it follows from Theorem $35.1$ in \cite{VA}, $\cap_{k=1}^\infty\cup_{|j|=k}\mathcal{D}_{i_j}$ is removable set for quaisconformal maps. Hence it follows from uniform boundedness of dilations we have the limit of 
$\tilde{f}_{\epsilon_1}=\lim f_{\epsilon_1,k}$ is quasiconformal map on $\mathbb{C}.$ \par
In addition, by our construction we have 
$\lim_{l\to\infty}\cap^l_k\cup_{|j|=k}f_{\epsilon_1,k}(\mathcal{D}_{i_j})=\lim_{l\to\infty}\cap^l_k\cup_{|j|=k}\mathcal{D}^{\epsilon_1}_{i_j}
\subset\cap_{k=1}^\infty\cup_{|j|=k}\mathcal{D}_{i_j}$
which implies that,  $\tilde{f}_{\epsilon_1}(\Lambda_\G)\subset\Lambda_\G$. This implies Hausdorff dimension $\mathfrak{D}_{\epsilon_1}$ of 
$\tilde{f}_{\epsilon_1}(\Lambda_\G)$
is $\mathfrak{D}_{\epsilon_1}\le\mathfrak{D}_\G.$\par
We set $\G_1=\tilde{f}_{\epsilon_1}\G\tilde{f}^{-1}_{\epsilon_1}.$ So we have a path $\G_s=\tilde{f}_s\G\tilde{f}^{-1}_s$ for $s\le\epsilon_1$ such that $\mathfrak{D}_{\G_s}\le\mathfrak{D}_\G$ for $s\le\epsilon_1.$
Now we repeat the above construction for $\G_1$ in place of $\G$ to get $\G_2$ and so on. By this process we have constructed sequence of $\{\G_n\}$ and a path of $\G_s$ such that $\mathfrak{D}_{\G_s}\le\mathfrak{D}_{\G}$ for $s\le n$. Since $\mathcal{D}^n_{i_j}$ converge to points as $n\to\infty$, we have 
$\mathfrak{D}_{\G_n}\to 0$, hence by continuity of Hausdorff dimension we have for sufficiently large $s$ such that $\mathfrak{D}_{\G_s}<\lambda$ 
where $\lambda>0$ is given by Theorem \ref{HouY} .  Hence it follows from Theorem \ref{HouY}, we have path $\G_s\subset\mathfrak{J}^1_g$ such that $\G_t\in\mathfrak{J}_{g,o}$ for some large $t$.

\end{proof}

\section{Schottky Space and rectifiable curves}
Let $\{\G_n\}$ be a sequence of classical Schottky groups such that $\G_n\to\G.$ Denote $QC(\mathbb{C})$ space of quasiconformal maps on $\mathbb{C}$. It follows from quasiconformal deformation theory of Schottky space, for
$\G_c$ classical Schottky group we can write ,
\[\mathfrak{J}_g=\{f\circ\gamma\circ f^{-1}\in\text{PSL}(2,\mathbb{C})|  f\in QC(\mathbb{C}), \gamma\in\Gamma_c\}/\mbox{PSL}(2,\mathbb{C}).\]

\noindent{\bf Notations:} 
Set $\mathfrak{H}$ to be the collection of all classical Schottky groups of Hausdorff dimension $\le\lambda$ for some $\lambda<1.$
\\

\par
Note that there exists a sequence of quasiconformal maps $f_n$ and $f$ of $\mathbb{C}$ such that, we can write
 $\G_n=f_n(\G)$ and $\G=f(\G_c).$ Here we write $f(\G):=\{f \circ g\circ f^{-1}| g\in \G\}$ for a given Kleinian group $\G$ and quasiconformal map $f.$\par
 
 Schottky space $\mathfrak{J}_g$ can also be considered as subspace of $\mathbb{C}^{3g-3}.$ This provides $\mathfrak{J}_g$ analytic structure as 
 $3g-3$-dimensional complex analytic manifold. 
 
 \begin{prop}\label{domain}
 Let $\{\G_n\}$ be a sequence of Schottky groups with $\G_n\to\G$ to a Schottky group $\G.$ Let $\mathcal{F}$ be a fundamental domain of $\G.$ There exists
 a sequence of fundamental domain $\{\mathcal{F}_n\}$ of $\G_n$ such that $\mathcal{F}_n\to\mathcal{F}.$ 
 \end{prop}
 
 \begin{proof}
 Let $\cup_{i=1}^{2g}\mathcal{C}_i=\partial\mathcal{F}$ be the Jordan curves which is the boundary of $\mathcal{F}.$ 
 Set $\mathcal{C}_n=f^{-1}_n(\cup_{i=1}^{2g}\mathcal{C}_i).$ Then $\mathcal{C}_n$ is the boundary of a fundamental domain of $f^{-1}_n(\G).$
 Hence we have a fundamental domain $\mathcal{F}_n$ of $\G_n$ defined by $\mathcal{C}_n$ with $\mathcal{F}_n\to\mathcal{F}.$
 \end{proof}

 \begin{lem}\label{sequence}
Suppose $f_n\to f$ and $\G=f(\G_c).$ Every quasi-circle with bounded length of $\G_n=f_n(\G_c)$ is in $\mathscr{R}(S^1,W)$ for some compact $W.$
 \end{lem}
 \begin{proof}
 Let $\eta_n$ be a quasi-circle of $\G_n.$ Note  $\Lambda_{\G_n}\subset\eta_n.$ Since limit set $\Lambda_{\G_n}\to\Lambda_{\G},$  and limit set $\Lambda(\G)$
 is compact, and $\eta_n$ is rectifiable, we can find some compact set $W\supset \cup_n\Lambda_{\G_n}\cup\Lambda_{\G}.$ 
 \end{proof}
 
Given $\{\G_n\}\subset\mathfrak{H}$ with $\G_n\to\G$, 
let $E_{\G_n}$ denote the collection of all bounded length quasi-circles of $\G_n.$ We define $\overline{\cup E_{\G_n}}$ to be the closure of $\cup E_{\G_n}$. 
Note that one can always get a non-simple closed curves given by non-simple generating curves, hence we need to remove all these \emph{trivial} singular curves in $\overline{\cup E_{\G_n}}$, i.e. curves that is generated by some non-simple generating curves, which we call them \emph{trivial singular curves}. 
Let $\mathscr{E}$ denote the collection of all trivial singular curves in $\eta\in\partial{\overline{\cup E_{\G_n}}}$, i.e. curves which have a singularity (non-simple)\emph{generating curve} of $\eta.$ We define $\Psi_\G=\partial\overline{\cup E_{\G_n}}-\mathscr{E}.$
\par

We define $\mathscr{O}(\eta)$, open sets about $\eta\in\Psi_\G$ in relative topology given by $O(\eta)\cap\Psi_\G$ for some open set $O(\eta)\subset\mathscr{R}(S^1,W).$ \par

Let $\G$ be a Schottky group. Let $\mathcal{F}$ be a fundamental domain of $\G.$ For $\mathscr{O}(\eta_\G)$ of $\eta_\G\in\Psi_\G$, and suppose every element 
is quasi-circle, and let $\zeta_\xi$ denote a generating curve of  $\xi\in\mathscr{O}(\eta_\G)$ with respect to $\mathcal{F}.$ Then we have, $\mathscr{O}(\zeta)=\cup_{\xi\in\mathscr{O}(\eta_\G)}\zeta_\xi$ the collection of all
generating curves of the open set $\mathscr{O}(\eta_\G)$ gives a open set of generating curves of $\eta_\G.$ On set of collection of all generating curves of elements of $\Psi_\G$, we define the topology as $\xi_{\eta_n}\to\xi_\eta$, if and only if $\eta_n\to\eta$ for $\eta_n,\eta\in\Psi_\G.$
\par
\begin{rem}
 We will sometime denote by $\eta_\infty\in\partial\Psi_\G$ a curve which is the limit of rectifiable quasi-circles of $\{\G_n\}.$
 \end{rem}
 \begin{prop}\label{seq-curve}
Let $\{\G_n\}\subset\mathfrak{H}$ with $\G_n\to\G.$ 
Then every bounded length quasi-circle $\eta_\G$ of $\G$ is in $\Psi_\G.$ In addition, if $\eta_n$ are linear then $\eta_\G$ is  linear quasi-circles of $\G.$  \end{prop}
 
\begin{proof}
Let $\G_n=f^{-1}_n(\G).$ Note that since $\mathfrak{D}_{\G_n}\le\lambda$, we have $\mathfrak{D}_{\G}\le\lambda<1.$
Define $\eta_n=f^{-1}_n(\eta_\G),$ for all $n.$ Then $\{\eta_n\}$ is a sequence of Jordan closed curves.
It follows from Proposition \ref{marking}, we have a generating curve $\zeta$ of $\eta_\G.$ So $\eta_\G=\Lambda_\G\cup\cup_{\g\in\G}\g(\zeta),$  and we have
$\eta_n=f^{-1}_n(\Lambda_\G)\cup f^{-1}_n(\cup_{\g\in\G}\g(\zeta)).$ Since $\Lambda_{\G_n}=f^{-1}_n(\Lambda_\G)$ and 
$f^{-1}_n(\cup_{\g\in\G}\g(\zeta))=\cup_{\g\in \G} f^{-1}_n\g f_n (f^{-1}_n(\zeta))$ so its $\cup_{\g_n\in\G_n}\g_n(\zeta_n),$ where $\zeta_n=f^{-1}_n(\zeta).$ Hence
$\zeta_n$ is a generating curve of $\eta_n$ which are quasi-circles of $\G_n$. Denote by $\zeta_\G$ a generating curve for $\eta_\G$. Since $\zeta$
is rectifiable curve, modify $\zeta_n$ if necessary,  we can assume $\{\zeta_n\}$ are rectifiable curves. \par
Let $z\in\zeta_n$, since $\mathfrak{D}_{\G_n}\le\lambda$, we have the $1$-dimension Hausdorff measure $\mu^1(\eta_n)$:
\[\mu^1(\eta_n)\asymp\mu^1(\zeta_n)\sum_{\g\in\G_n}|\g'(z)|<\infty,\]
where $\g'$ is the derivative of $\g.$
Hence $\{\eta_n\}$ are rectifiable quasi-circles. It follows that there exists $c>0$ such that $\mu^1(\eta_n)<c\mu^1(\eta_\G)$ for large $n$, hence 
$\{\eta_n\}$ are bounded quasi-circles of $\mathscr{R}(S^1,W)$,  we have $\eta_n\to \eta_\G$ and  $\eta_\G\in\Psi_\G.$
Finally, if $\eta_n$ are linear then $\zeta_n$ are linear and since Mobius maps preserves linearity, we have $\eta_\G$ is linear. 
\end{proof}

 \begin{defn}[Good-sequences] \label{good}
 For a given sequence $\{\eta_n\}$ of quasi circles of $\{\G_n\}\subset\mathfrak{H}$ with $\G_n\to\G$, we say $\{\eta_n\}$ is a 
 \emph{good-sequence} of quasi-circles if, it is convergent sequence and $\eta_\infty$ is a quasi-circle of $\G.$ We also call $\{\eta_n\}$ (non)transverse good-sequence if all $\eta_n$ are also (non)transverse.
 \end{defn}
 
 \begin{lem}[Existence]\label{good}
 Let $\{\G_n\}\subset\mathfrak{H}$ be a sequence of Schottky groups with $\G_n\to \G.$ There exists a good-sequence 
 $\{\eta_n\}$ of quasi-circles of $\{\G_n\}.$ In addition, if $\{\eta_n\}$ is also (non)transverse then $\eta_\infty$ is (non)transverse.

 \end{lem}
 
 \begin{proof}
Let $f_n$ be quasi-conformal maps such that $\G_n=f_n(\G).$ Let $\eta_\G$ be a quasi-circle of $\G.$ Then $\eta_n=f_n(\eta_\G)$ is a good-sequence of quasi-circles. The (non)transverse property obviously is preserved. 

 \end{proof}
 
 \begin{cor}[Linear-invariant]\label{linear-invariant}
 Let $\{\eta_n\}$ be a good-sequence of linear quasi-circles of $\{\G_n\}\subset\mathfrak{H}.$  $\eta_\infty$ is a linear quasi-circle of $\G.$ 
 \end{cor}
 
 \begin{proof}
 Linearity is obviously preserved at $\eta_\infty.$
 \end{proof}
 
 \begin{cor}[Tranverse-invariant]\label{transverse}
  Let $\{\eta_n\}$ be a good-sequence of quasi-circles of $\{\G_n\}\subset\mathfrak{H}.$  Then $\eta_\infty$ is a (non)transverse linear quasi-circle of $\G$ if and only if $\{\eta_n\}$ is a (non)transverse. 
 \end{cor}
 \begin{proof}
 \end{proof}

\begin{lem}[Open]\label{open}
Let $\{\G_n\}\subset\mathfrak{H}$. Let $\G_n\to \G$ be a Schottky group. Let $\eta_\G$ be a quasi-circle of  $\G.$ Then
there exists a relative open neighboredood $\mathscr{O}(\eta_\G)$ of $\eta_\G$ such that every element of $\eta_\G$ in $\mathscr{O}(\eta_\G)$ is a quasi-circle of
$\G.$ 
\end{lem}
\begin{proof}
Let $\mathcal{F}$ be a fundamental domain of $\G.$ Let $\xi$ be the generating curve of $\eta_\G$ with respect to $\mathcal{F}.$ 
Let $\mathcal{F}_n$ be the fundamental domain of $\G_n$ with $\mathcal{F}\to\mathcal{F}_n.$ By Proposition \ref{seq-curve}, we have $f_n:\eta_\G\to\eta_n.$ 
Let $\mathcal{U}(\xi)$ be a small open neighborhood about $\xi$ generating curve of $\eta_\G$ and Let 
$\mathcal{U}(\xi_n)$ to be the open set about the generating curve $\xi_n$ of $\eta_n$ given by $f_n(\mathcal{U}(\xi))$, and set $\mathcal{O}(\xi_n)=\mathcal{U}(\xi_n)\cap\Psi_\G.$ Then $\mathcal{O}(\xi_n)$ is open
sets of generating curves in $\Psi_\G$ (Figure 2). Let $\mathcal{O}(\xi_n)=f_n(\mathcal{O}(\xi))$ for $\mathcal{O}(\xi)$ a open set of generating curves for 
$\eta_{\G}.$ Let $\xi'_n\in f_n(\mathcal{O}(\xi_c))$ and $\eta_n'$ be generated by $\xi'_n.$ Assuming $\mathcal{O}(\xi)$ is sufficiently small neighborhood, we
have length $\ell(\eta_n')<c\ell(\eta_\G)$ for some small $c\ge 1$ for large $n$ and $\eta_n'$ generated by all $\xi'_n\in f_n(\mathcal{O}'(\xi)).$ 
Since $\eta_n\to\eta_\G$, we have $\eta'_\infty$ is quasi-circle of $\G$ with bounded length. This defines a open set which all elements are quasi-circles in $\Psi_\G$. \end{proof}
Note that the above lemma essentially states that, for  $\eta_n$  quasi-circles of $\G_n$ with $\eta_n\to\eta_\G$ and $\xi_n$ generating curves of $\eta_n$, we have for some small deformations of $\xi_n$ which generates $\eta'_n$ and converges in $\Psi_\G$ is quasi-circles of $\G$.

\begin{figure}[ht!]
\centering
\includegraphics[width=35mm]{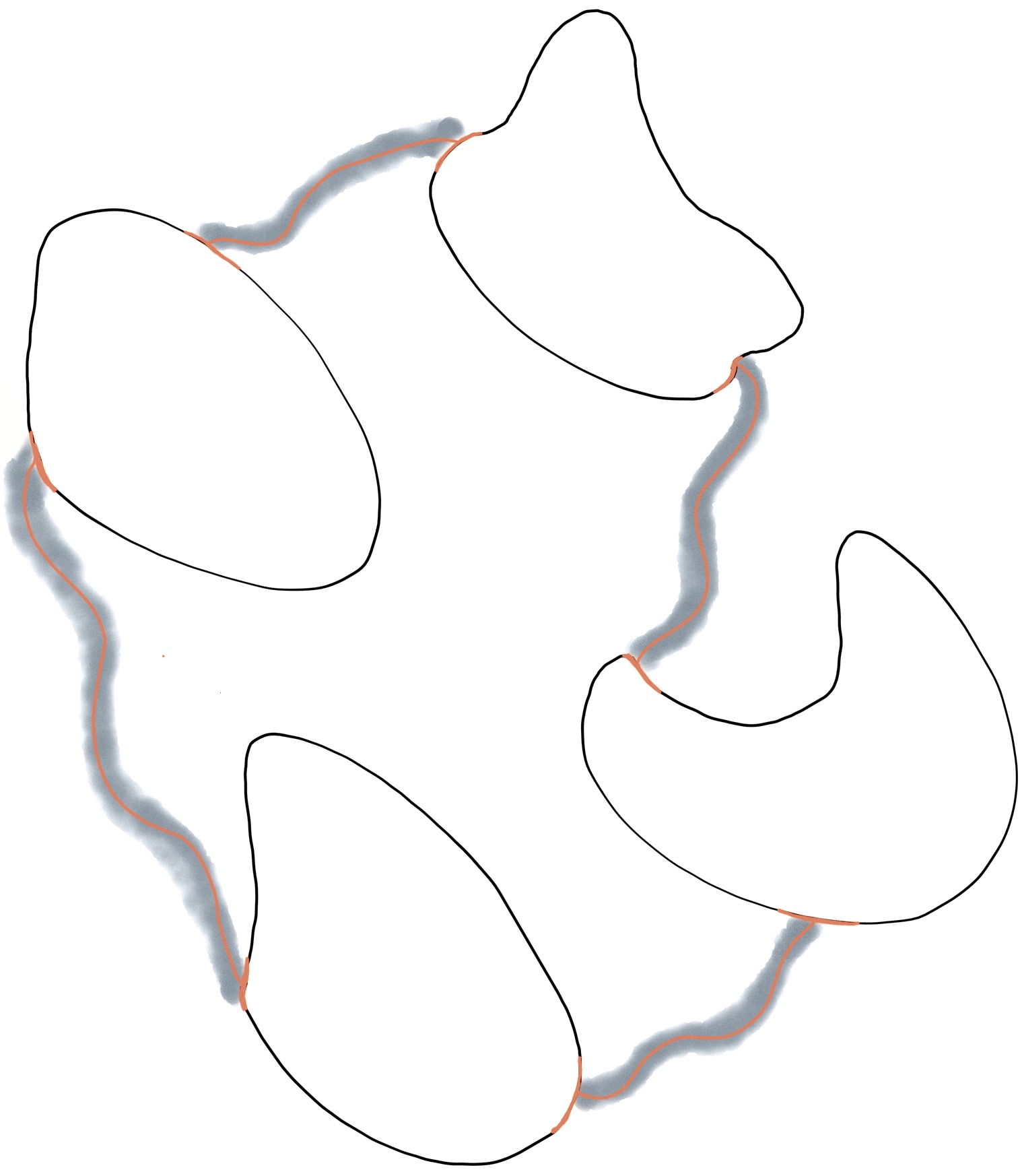}  
\caption{Open set of generating curves about $\zeta$ \label{overflow}}
\end{figure}
\begin{cor}\label{open-seq}
Let $\{\G_n\}\subset\mathfrak{H}$ with $\G_n\to\G$ a Schottky group $\G.$ Let $\{\eta_n\}$ be a good-sequence of quasi-circles of $\{\G_n\}.$ Then
there exists an open neighborhood $\mathscr{O}(\eta_\infty)$ of $\eta_\infty$ such that every element of $\mathscr{O}(\eta_\infty)$ is a quasi-circle of
$\G.$
\end{cor}
 \begin{proof}
Follows from Lemma \ref{good} and Lemma \ref{open}.
\end{proof}
Next we analyze the formations of singularities for a given sequence of classical Schottky groups converging to a Schottky group. These types of singularities has been
studied in \cite{Hou1,Hou}.
\begin{lem}[Singularity]\label{singular}
Let $\{\G_n\}\subset\mathfrak{H}$ with $\G_n\to\G$ a Schottky group. Assume that $\G$ is non-classical Schottky group.  Then there exists $\eta_\G$ such that,
every $\mathscr{O}(\eta_\G)$ contains a deformation of $\eta_\G$ which is non-quasi-circle. 
\end{lem}
Here we say a closed curve is a \emph{non-quasi-circle}, if it's non-Jordan(contains a singularity point), or it's not $\G$-invariant.
\begin{proof}
For each $n$, let $\mathcal{F}_n$ be a classical fundamental domain of $\G_n.$ 
Given a sequence of classical fundamental domains $\{\mathcal{F}_n\}$ the convergence is consider as follows: $\{\partial\mathcal{F}_n\}$ is collection of 
$2g$ circles $\{c_{i,n}\}_{i=1,...,2g}$ in the Riemann sphere $\overline{\mathbb{C}}$, pass to a subsequence if necessary, then $\lim_n c_{i,n}$, is either a point or a circle. 
We say $\{\mathcal{F}_n\}$ convergents to $\mathcal{G}$, if $\mathcal{G}$ is a region that have boundary consists of 
$\lim_n c_{i,n}$ for each $i$, which necessarily is either a point or a circle. Note that $\mathcal{G}$ is not necessarily a fundamental domain, nor it's necessarily connected.\par
Let $\lim\mathcal{F}_{n}=\mathcal{G}.$ By assumption that $\G$ is not a classical Schottky group, we have $\mathcal{G}$ is not a classical fundamental domain of $\G.$ We have $\partial\mathcal{G}$ consists of circles or points. However these circles may not be disjoint. More precisely, we have the following possible degeneration of circles of $\partial\mathcal{F}_n$ which gives $\partial\mathcal{G}$ of at least one of following singularities types:
\begin{itemize}
\item{Tangency}:
Contains tangent circles. 
\item{Degeneration}:
Contains a circles degenerates into a point. 
\item{Collapsing}:
Contains two circles collapses into one circle. Here we have two concentric circles centered at origin and rest circles squeezed in between these two and
these two collapse into a single circle in $\mathcal{G}.$

\end{itemize}
\noindent $\partial\mathcal{G}$ contains: Tangency.\par
Let $p$ be a tangency point.
Let $\eta_\G$ be a quasi-circle of $\G$ which pass through point $p.$ Assume the Lemma is false, then all sufficiently small open neighborhood $\mathscr{O}(\eta_\G)$ contains only quasi-circles of $\G.$  It follows from Proposition \ref{sequence}, we have a sequence with $\eta_n\to\eta_\G$.
Let $\xi_n$ be the generating curves of $\eta_n$ with respect to $\mathcal{F}_n.$ Define $\phi_n$ as follows. Note that $\xi_n\cap\partial\mathcal{F}_n$ consists of points or linear arcs. We define $\zeta_n$ be a generating curve with $\zeta_n\cap\partial\mathcal{F}_n$ to consists of linear arcs, with one of its end point to be the point that converges into a tangency in $\partial\mathcal{G}.$  In addition we also require the arcs on $\partial\mathcal{F}_n$ that with a end points which converges to tangency point be be connect by a arc between the other end points. It is clear that every 
$\mathscr{O}(\eta_n)\cap\mathcal{F}_n$ contains a generating curve of this property. Let $\phi_n$ be the quasi-circle generated by $\zeta_n,$ and set
$\phi=\lim\phi_n$, then $\phi$ contains a loop singularity at a point of tangency (Figure 3). Hence $\phi$ is non-quasi-circle of $\G.$  Since every $\mathscr{O}(\eta_\G)$
contains such a $\phi$ we must have the lemma to be true for this type of singularity.\par
\begin{figure}[ht!]
\centering
\includegraphics[width=35mm]{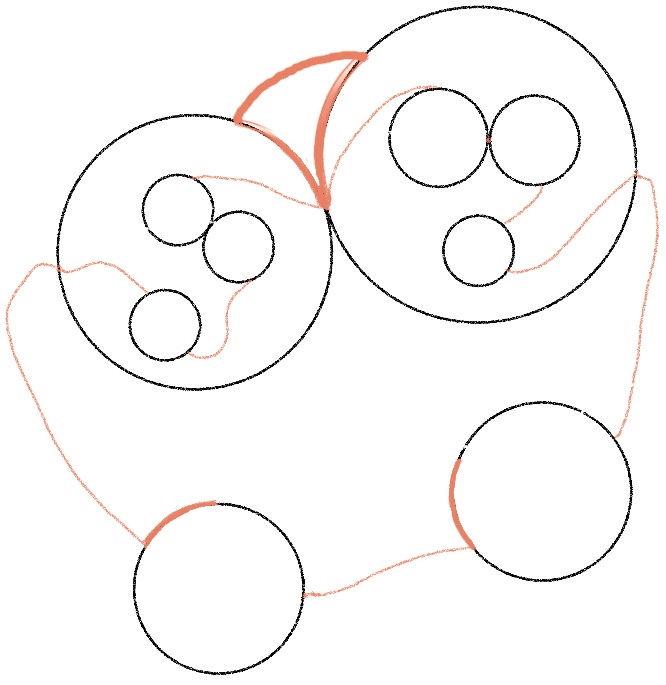}  
\caption{Tangency singularity \label{overflow}}
\end{figure}
\noindent  $\partial\mathcal{G}$ contains: Degeneration.\par
In this case, we must have any quasi-circle $\eta_n$ pass through a degeneration point. Note that we can't have all circles degenerates into a single point which would collapses quasi-circles to a point.  
To see this, let $\g_{i,n}\to\g_i$ be the generating set of $\G_n$ corresponding to circles $\{c_{i,n}\}$ and $\g_i$ it's limit in $\G.$ Since by assumption $\G$ is
Schottky group, so $\g_i$ must all be loxodromic, hence distinct fixed points. But if all circles degenerates into a single point, and since all fixed points of 
$\g_{i,n}$ are contained within disjoint regions bounded by circles $\{c_{i,n}\}$ then, we can't have all of $\{\g_i\}$ fixed points distinct, a contradiction. \par
In general, even $\G\in\partial\overline{\mathfrak{J}_g}$ we still can't have all circles degenerates into a single point. Since degenerating into a single point would imply $\G$ either is not discrete or not free group. But element of $\partial\overline{\mathfrak{J}_g}$ is either cusps or geometrically infinite group, and both
types are discrete and free group.
\par
There are two possibilities to have degenerate points. 
Case (A): two circles merge into a single degenerate point $p$. Case (B): A circle degenerates into a point on a circle. \par
Consider (A). In this case, any quasi-circles $\eta_\infty$ will have two possible properties, either there is a point $q$ on some circle of $\partial\mathcal{G}$ such that every quasi-circle
must pass through $q$, or two separate arcs of $\eta_\infty$ meet at $p$. The second possibility implies there exists a loop singularity at $p$, hence $\eta_\infty$ can not be a Jordan curve (Figure 4). Therefore we only have the first possibility for $\eta_\infty.$ But it follows from Proposition \ref{seq-curve}, all rectifiable quasi-circles of $\G$ is the limit
of some sequence of quasi-circles of $\G_n$, and hence all must pass through $q$. However $q$ is not a limit point, hence we can have some quasi-circle not passing through $q,$ a contradiction.\par
\begin{figure}[ht!]
\centering
\includegraphics[width=65mm]{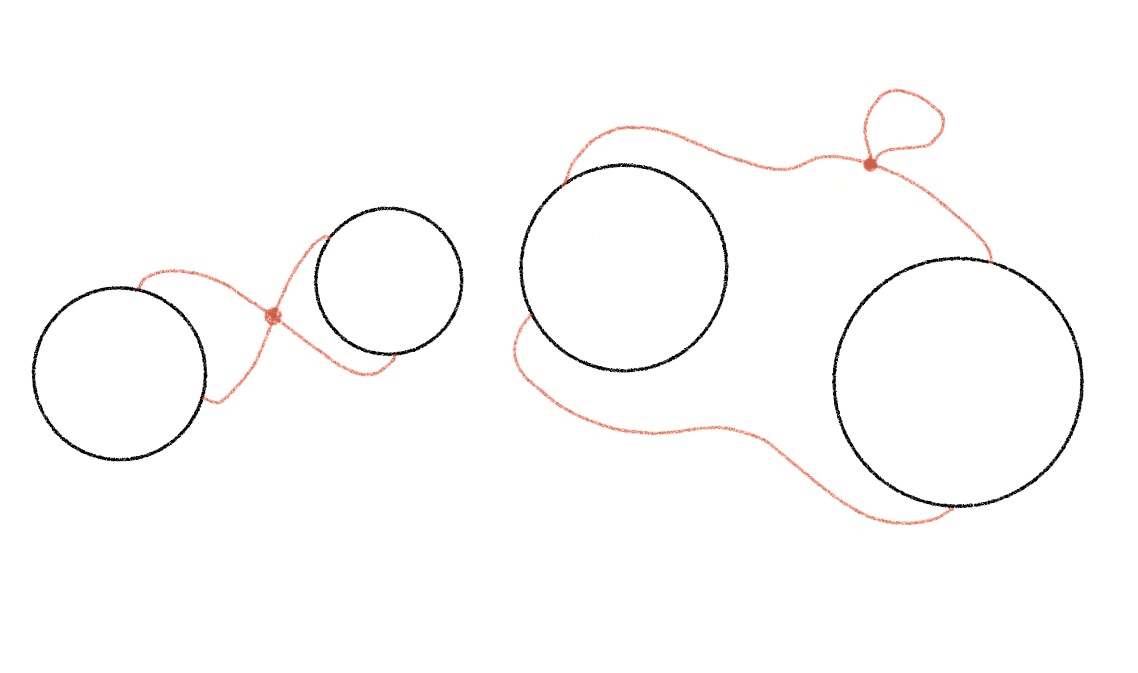}  
\caption{Degenerate singularities \label{overflow}}
\end{figure}

Now consider (B). Here we can assume that there exists at least two circles that do not degenerates into points. Otherwise, we will have the third type (collapsing) singularity, which we will consider next. Having some circles degenerates into a point on to a circle at $p$,  we have a sequence of  quasi-circles $\eta_n$ which passes through $p.$ This is given by the generating curves that have curves connecting the degenerating circle to point $p_n\to p$ converging to linear arc intersecting
orthogonally at boundary. But any neighboring quasi-circle of $\eta_n$  will converges to a $\eta_\infty$ with a loop singularity at $p,$ which is not a Jordan curve
(Figure 5). \par
\begin{figure}[ht!]
\centering
\includegraphics[width=30mm]{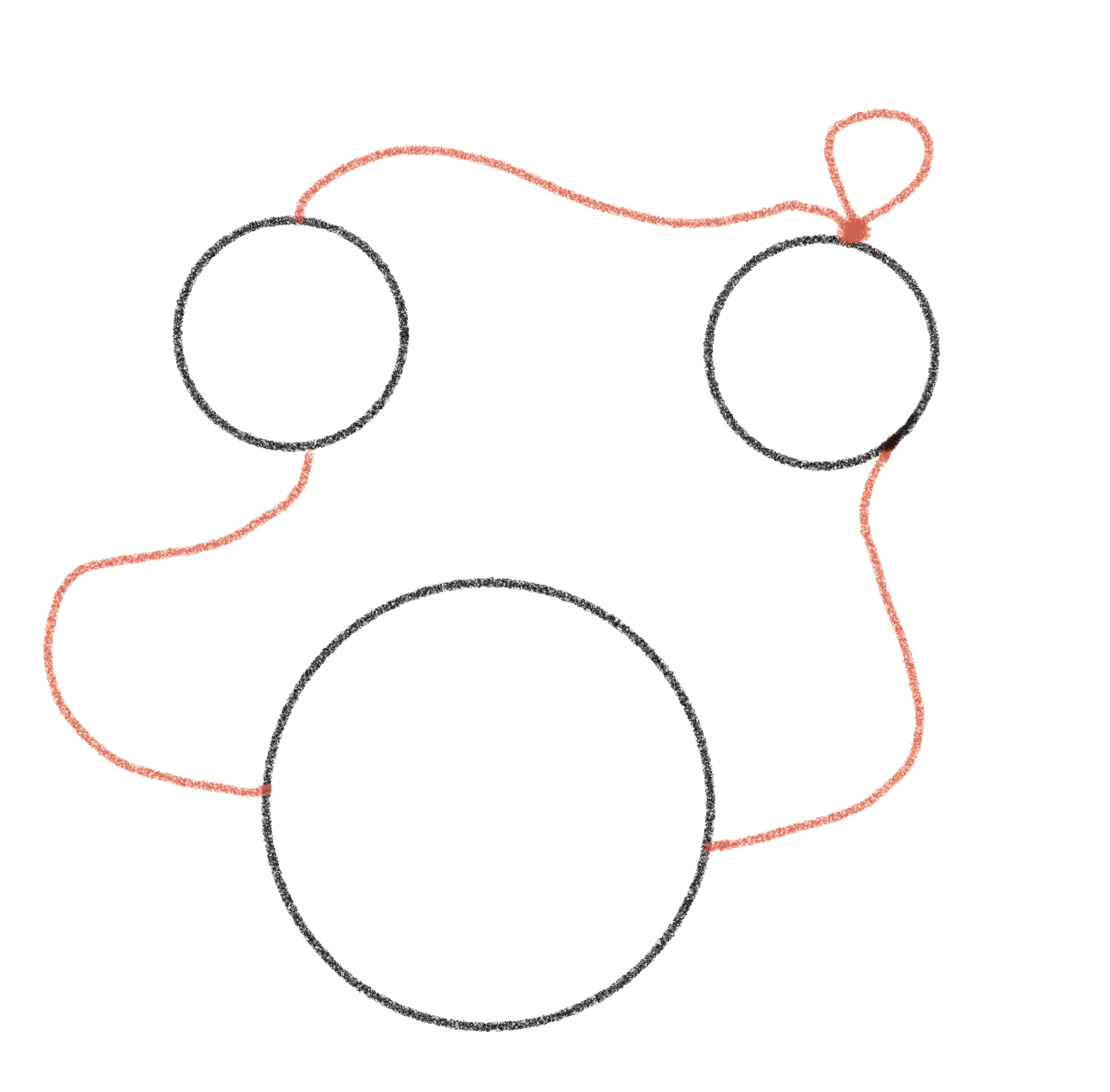}  
\caption{Degenerate singularities \label{overflow}}
\end{figure}

Finally, we note that if we have a degeneration point $p$ which is not a limit point, then there exists a neighboring quasi-circle of $\eta_\G$ that misses the point $p.$ Hence any $\mathscr{O}(\eta_\G)$ will contains some curve which is the limit of a sequence of quasi-circles that do not pass through the point. This gives a non-quasi-circle in $\mathscr{O}(\eta_\G).$\\
\\
\noindent  $\partial\mathcal{G}$ contain: Collapsing.\par
Let $\mathcal{C}$ denote the collapsed circle, i.e $\mathcal{F}_n\to\mathcal{C}.$ We will show then $\Lambda_\G\subset\mathcal{C}.$ Assume the contrary and  
suppose that there exists a fixed point not on 
$\mathcal{C}.$ Then there are infinitely many elements with fixed points not on $\mathcal{C},$  since limit set is perfect set. \par
Let $\g_n\in\G_n$ with $\g_n\to\g.$  
Since $\g_n(\mathcal{F}_n)\to\g(\mathcal{C}),$ we must have either $\g(\mathcal{C})=\mathcal{C}$ or 
$\g(\mathcal{C})\cap\mathcal{C}=\emptyset.$ 
By assumption we have a fixed point of a element of $\G$ is $\not\in\mathcal{C},$ which implies infinitely many fixed points $\not\in\mathcal{C}.$ So take three distinct fixed points $a,b,c$ which are respectively attractive fixed points of $\g_a,\g_b,\g_c\in\G$ such that $a,b,c\not\in\mathcal{C}.$ Then $\g^k_\alpha(\mathcal{C})\to\alpha$ for each $\alpha\in\{a,b,c\}$ and $k\to\infty$. Let $\g^k_{n,\alpha}\in\G_n$ with $\g^k_{n,\alpha}\to\g^k_\alpha.$ Since $\mathcal{F}_n\to\mathcal{C}$ so we have 
$\g^k_{n,\alpha}(\mathcal{F}_n)\to\g^k_\alpha(\mathcal{C})$ for each $k.$ Now since $\g^k_\alpha(\mathcal{C})$ converges to distinct fixed points $a,b,c$ so by choose $k$ large enough we have for sufficiently large $n$, $\g^k_{n,\alpha}(\mathcal{F}_n)$ are contained within distinct regions bounded by circles of $\partial\beta_n(\mathcal{F}_n)$ for some $\beta_n\in\G_n.$ To show this, we first choose a element $\beta_n'\in\G_n$ so that fixed point $a$ is contained in a disk bounded by $\partial\beta'_n(\mathcal{F}_n)$ and fixed points $b,c$ are bounded in a region disjoint from the disk. Then choose $\beta_n''\in\G_n$ so that $b,c$ fixed points lie within disjoint disks bounded by 
$\partial\beta_n(\mathcal{F}_n),$ where $\beta_n=\beta''_n\beta_n'.$ Hence for large $k$, we have $\g^k_\alpha(\mathcal{C})$ for $\alpha\in\{a,b,c\}$ will lies in disjoint regions bounded by circles of $\partial\beta_n(\mathcal{F}_n)$, for sufficiently large $n$, see Figure 6.
\begin{figure}[ht!]
\centering
\includegraphics[width=58mm]{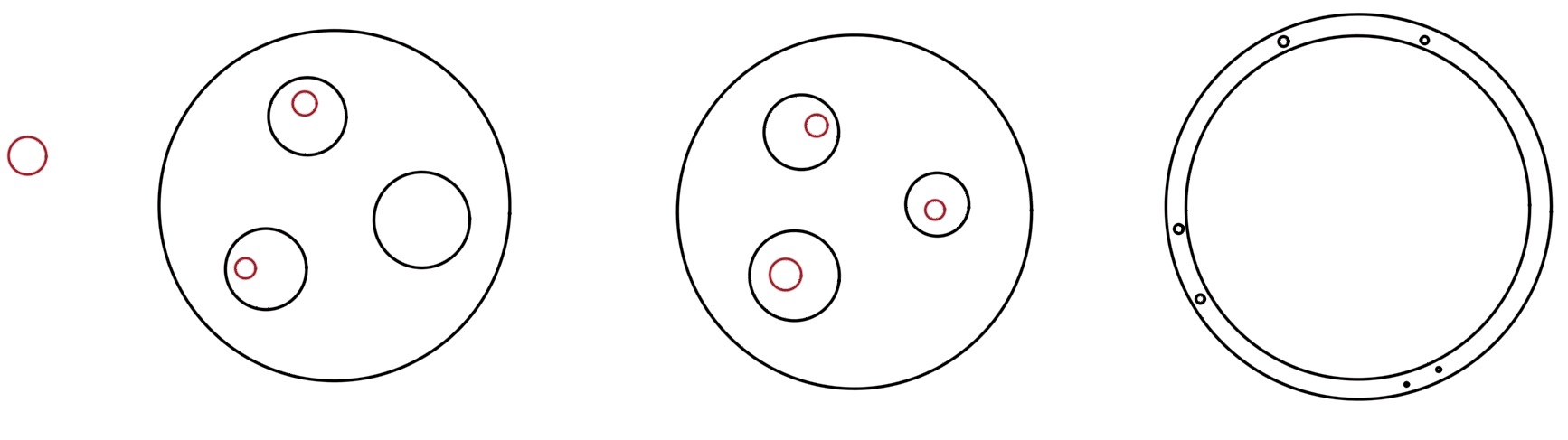}  
\caption{First figure have a $\g^k_\alpha(\mathcal{C})$(in red) outside all circles which forbid collapsing. Second have all three $\g^k_\alpha(\mathcal{C})$ in a big circle
which forbid collapsing. The third don't contain any and so collapses. \label{overflow}}
\end{figure}\par
Let $\beta\in\G$ be the limit of $\beta_n.$ Since $\beta_n(\mathcal{F}_n)\to\beta(\mathcal{C})$ we have $\beta_n(\mathcal{F}_n)$ collapsing to a circle. However,
$\g^k_\alpha(\mathcal{C})$ do not collapse to points for fixed $k.$ Hence $\beta_n(\mathcal{F}_n)\not\to\beta(\mathcal{C})$, a contradiction.

\par 
Hence all fixed points of elements of 
$\G$ are $\subset\mathcal{C},$ which implies $\Lambda_\G\subset\mathcal{C}.$ Since $\G$ is Schottky group, we must have $\G$ is Fuchsian schottky group.
By \cite{Button}, $\G$ is classical Schottky group, a contradiction.

\end{proof}
\begin{lem} 
$\eta_\G$ in Lemma $3.11$ is non-trivial singular $\G$-invariant closed curve.
\end{lem}
\begin{proof}
Consider tangency. In this case we will see that we have either: $(a)$ tangencies are removable i.e. we can choose classical domain without tangency, or
$(b)$ one of the tangencies is a fixed point, or $(c)$ there exists $\eta_\G$ such that every singularities have \emph{four} branches arcs (degree four singularity), such that $\eta_\G$ curve is formed by linking collection of closed graphs of girth $3$ with every vertices of degree four in the graph. We will see that in (b) and (c) 
$\eta_\G$ is \emph{not} trivial singular curve. This can be seen as follows.\par 
First for simplicity suppose rank two $<\alpha,\beta>$ with corresponding circles $\{C_\alpha,C_\alpha',C_\beta,C_\beta'\}$ with tangency. If either $C_\alpha,C_\alpha'$ or $C_\beta,C_\beta'$ is tangent then we have a fixed point of the tangency. \par 
Now suppose
we have just one of tangency from either $C_\alpha,C_\beta$ or $C_\alpha,C_\beta'$ both not both, say $C_\alpha,C_\beta$. Then we can just shrink $C_\alpha$ 
the radius a bit and enlarge $C_\alpha'$ proportionally so that it remains disjoint from rest of circles. Better to way to see this is by conjugating $C_\alpha,C_\alpha'$
into circles centered on origin, then one can easily see we can remove the tangency by choose another set of circles which gives classical domain, see Figure $7$. \par
Next suppose
both $C_\alpha,C_\beta$ and $C_\alpha,C_\beta'$ are tangent. Then the tangency is either fixed point of $\alpha^{-1}\beta$ or it's not a fixed point. 
Suppose the latter, then the singularity constructed in the proof of Lemma 3.11 gives $\bigtriangledown$-type singularity with bottom vertex the tangency of
$C_\alpha,C_\beta$. Similarly $C_\alpha,C_\beta'$ produce same type of singularity. We choose the top vertices of $\bigtriangledown$ to be images of tangency under $\alpha,\beta $, since by assumption the tangency are not fixed points. We connect the two pair tangent circles by straight arcs correspondingly. Then these
triangular curves tessellating into graph of $\bigtriangledown\bigtriangleup$ attached through $\alpha,\beta$ maps, and converges to limit points of $\Gamma$ in the curve, hence we have a closed graph. In addition, it's easy to see that every vertex of the graph is of degree four with girth three in $\eta_\G$. Hence $\eta_\G$ is formed by link (images of straight arcs) of collections of degree four closed graphs of girth $3$, see Figure $8$.
\begin{figure}[ht!]
\centering
\includegraphics[width=35mm]{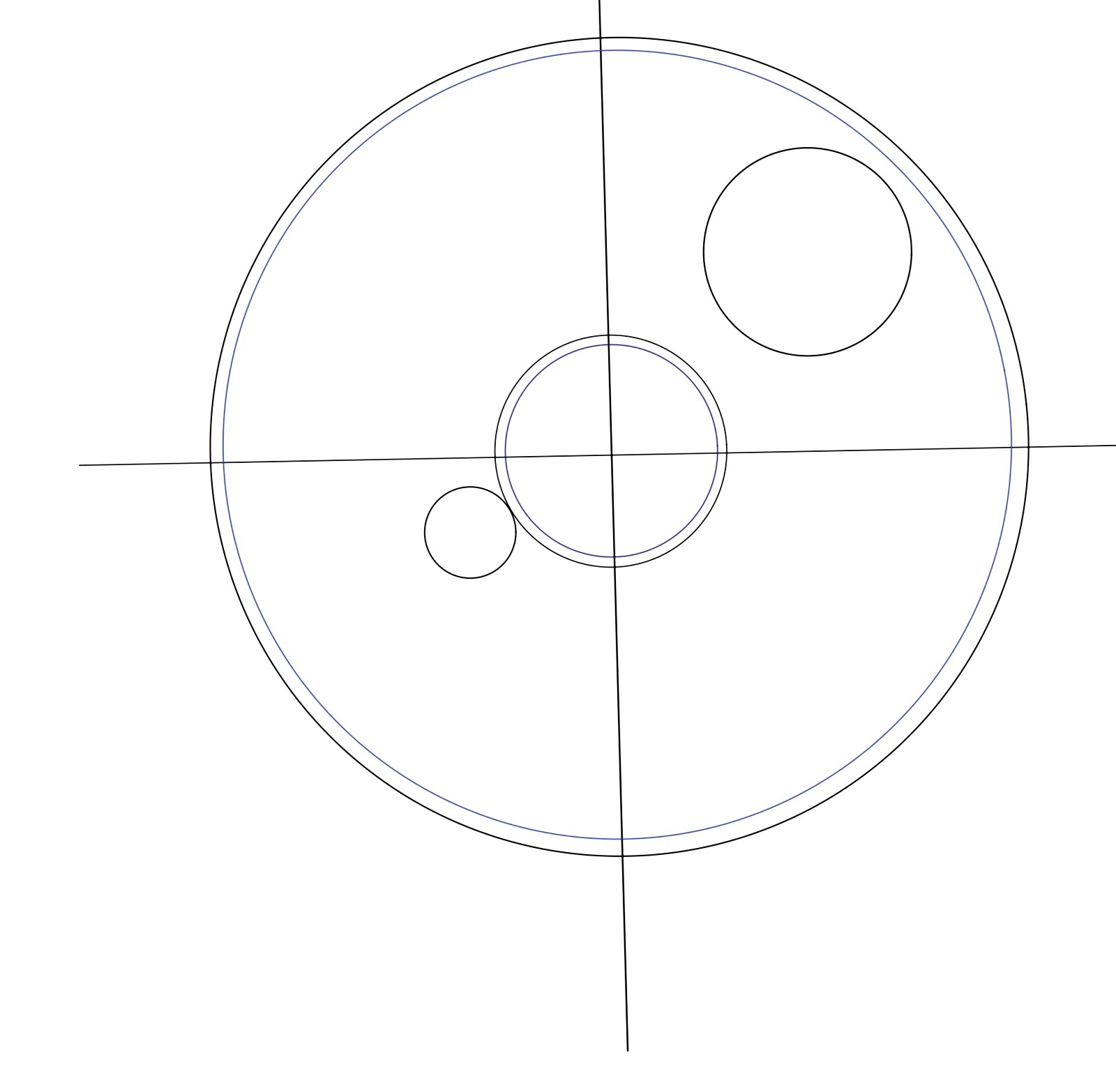}  
\caption{removable tangency \label{overflow}}
\end{figure}
\begin{figure}[ht!]
\centering
\includegraphics[width=38mm]{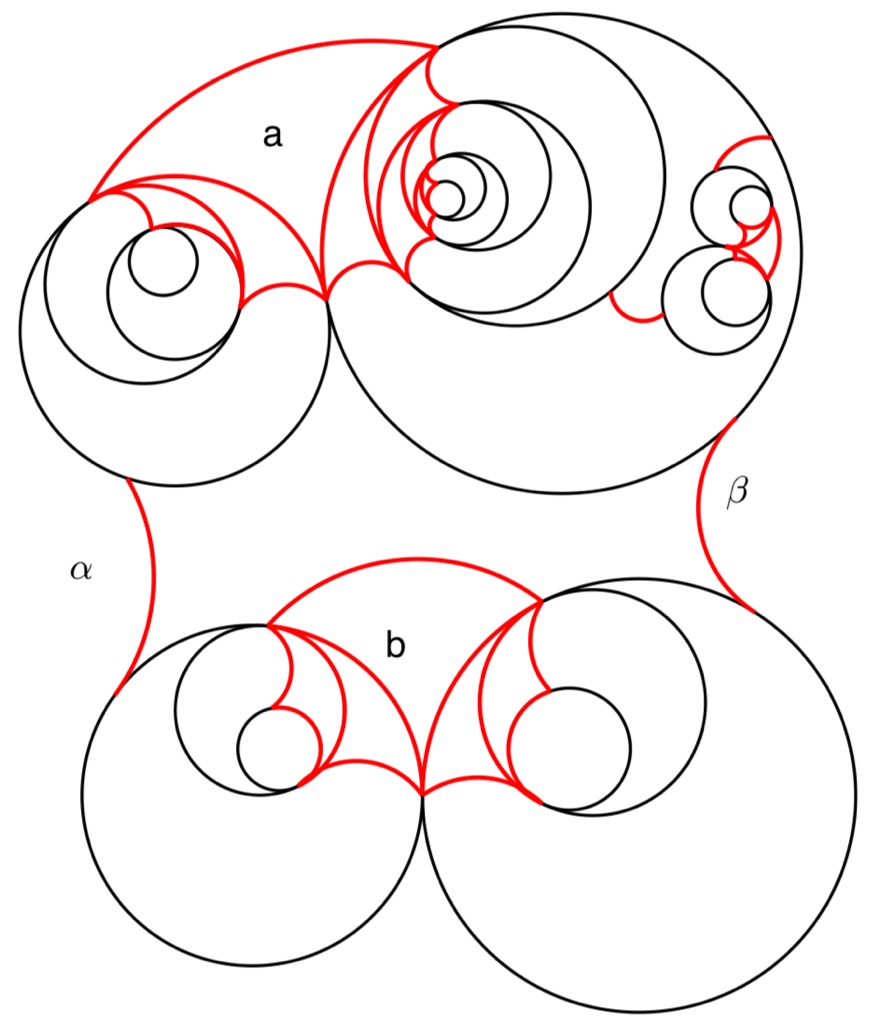}  
\caption{tangency singularities \label{overflow}}
\end{figure}
\begin{figure}[ht!]
\centering
\includegraphics[width=39mm]{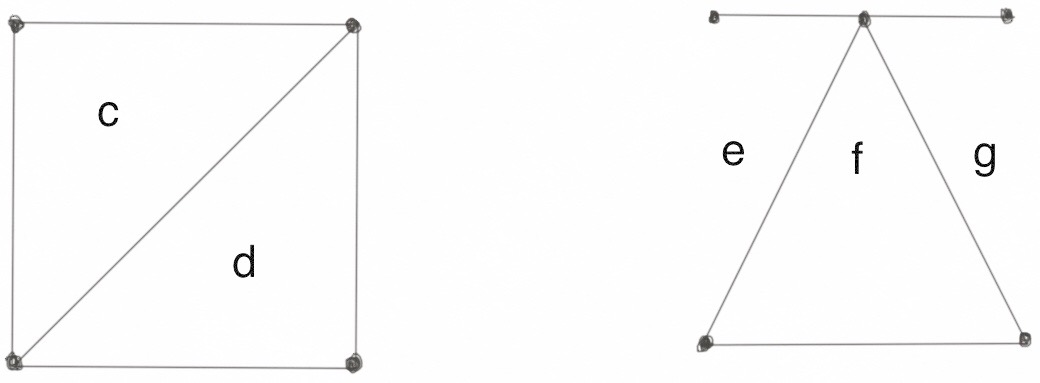}  
\caption{singular generating curves \label{overflow}}
\end{figure}

Now we will see that $\eta_\G$ is \emph{not trivial singular curve.} Suppose the contrary, then we can choose some Schottky domain so that $\eta_\G$  is given by singular generating curves $\xi$ which are limit of generating curves of quasi-circles.
By edge and degree restrictions and girth of the graph, we have possible $\xi$ given by combinations of: $\bigboxslash$ , $\bar{\bigtriangleup}$ see Figure $9$. We label the triangular regions according in Figure $9$. In addition we label the tangency triangular region in Figure $8$ by $a,b$. There are three possible combinations of $\xi$: $\bigboxslash\bigboxslash$, 
$\bar{\bigtriangleup}\bar{\bigtriangleup}$, $\bigboxslash\bar{\bigtriangleup}.$ We show that none of them can form $\eta$ as $\Gamma$ invariant curve.\par
Denote generators of the new Schottky domain by $\gamma,\rho$. Consider $\bigboxslash\bigboxslash$. We denote the triangular region of the other 
$\bigboxslash$ by $c',d'$ respectively. Since they generate $\eta$, there exists $\psi\in\Gamma$
such that $c=\psi(a), d=\psi\beta(b), \rho(c')=\psi\beta\alpha^{-1}(a),\rho(d')=\psi\beta\alpha^{-1}\beta(b), \rho\gamma^{-1}(c)=\psi\beta\alpha^{-1}\beta\alpha^{-1}(a).$ This implies that $\rho\gamma^{-1}(c)=\psi(\beta\alpha^{-1})^2\psi^{-1}(c).$ Hence $\psi^{-1}\rho\gamma^{-1}\psi=(\beta\alpha^{-1})^2.$ 
After simple computations this equality implies either $\rho=\gamma^{-1}$, or $r\rho^{-1}r=\gamma$ for some $r\in\Gamma.$ However since $\rho,\gamma$ are free generators, this is impossible.  
Hence this relation implies $\rho,\gamma$ are not generators expressible in terms of $\beta,\alpha.$ \par
Similarly for $\bar{\bigtriangleup}\bar{\bigtriangleup}$ we denote the triangular regions of second $\bar{\bigtriangleup}$ by $f'$ correspondingly. Then we have, $f=\phi(a), \rho(f')=\phi\beta\alpha^{-1}(a),\rho\gamma^{-1}(f)=\phi\beta\alpha^{-1}\beta\alpha^{-1}(a)$ for some $\phi\in\Gamma.$ Hence we have $\rho\gamma^{-1}=\phi(\beta\alpha^{-1})^2\phi^{-1}.$ It follows from previous argument which implies $\rho,\gamma$ are not generators of $\Gamma$.\par 
Finally consider $\bigboxslash\bar{\bigtriangleup}.$ There exists $\lambda\in\Gamma$ where we have, $c=\lambda(a),d=\lambda\beta(b),\rho(e)=\lambda\beta\alpha^{-1}(a)$, 
$\rho(f)=\lambda\beta\alpha^{-1}\beta(b),\rho(g)=\lambda\beta\alpha^{-1}\beta\alpha^{-1}(a),\rho\gamma^{-1}(c)=\lambda\beta\alpha^{-1}\beta\alpha^{-1}\beta(b),
\rho\gamma^{-1}(d)=\lambda\beta\alpha^{-1}\beta\alpha^{-1}\beta\alpha^{-1}(a).$ Hence using these relations we have, 
$(\beta\alpha^{-1})^{-3}\lambda^{-1}\rho\gamma^{-1}\lambda=\lambda^{-1}\gamma\rho^{-1}\lambda(\beta\alpha^{-1})^2.$ 
Set $\chi=(\beta\alpha^{-1})^{-3}\lambda^{-1}\rho\gamma^{-1}\lambda.$ Then we have $\chi=\chi^{-1}\alpha\beta^{-1}.$ This implies $\chi^2=\alpha\beta^{-1}$, which is impossible.
\par
Hence it follows from above argument and the assumption of Lemma 3.11, $\Gamma$ is not classical Schottky group, we can only have the possibility of having tangency as fixed points. However, one can \emph{not} choose another Schottky domain of Jordan curves which would exclude fixed point outside these Jordan curves. Hence singularity at fixed point can not be trivialized by choose different domains i.e. they are \emph{non-removable}.\par
Similarly for $k$ circles with tangency. Following previous argument, we have either all pairs tangency or we can remove the tangency one by one as before by choose difference classical Schottky domain circles. For all pairs tangency we have $\gamma_i,\gamma_j$ gives tangency. It's sufficient to assume $i\not=j$ and  tangent point don't maps into other tangent point, otherwise it gives fixed point. Similarly as before they generates $\eta_\G$ with links of collection of closed graphs of degree four edges. It follows from previous argument, we can not have singular generating curves that gives $\eta_\G$. \par
Consider degeneracy. In this case we have circles degenerating into points. If we have a singularity at degenerated point which is limit point then we know it can not be trivialized by choose another Schottky domain. Now on the other hand, suppose we have all the degenerate points are not as limit points and, these singularities are trivial. Then $\eta_\G$ is given by singular generating curves with respect to some Schottky domain of $\Gamma$ bounded by disjoint Jordan curves 
$\mathscr{C}.$ This Schottky domain is the limit of sequence of Schottky domain bounded by Jordan curves $\mathscr{C}_n$ of $\Gamma_n$. Denote the corresponding sequence of quasi-circles by $\eta_n.$ Let $\mathscr{C}_{i,n}$ be the Jordan curves that bounds the part of the regions which encloses limit points of $C_{i,n}$ (the circles degenerating to points). Then with respect to $\mathscr{C}_n$ and since $\eta_n\to\eta_\Gamma$, we must have $\eta_\Gamma$ contains infinite numbers of singularities within $\mathscr{C}_{i,n}$. However since $C_{i,n}$ degenerates into points, we can only have finite collection of singularities within that region, which is contradiction. Hence we can not have $\eta_\Gamma$ given by singular generating curves.
\end{proof}

\section{Proof of Theorem \ref{main}}
\begin{proof}(Theorem \ref{main})\vspace{0.2cm}\\
First note that by Selberg Lemma we can just assume Kleinian group to be torsion-free.\par
Now note that if a Kleinian group $\G$ of $\mathfrak{D}_\G<1$ then it must be free. To show this, assume otherwise. Since $\mathfrak{D}_\G<1$ and $\G$ is purely loxodromic, it is convex-cocompact of second-kind. There exists an imbedded surface 
$\mathscr{R}=\Omega_\G/\G$ in $\mathbb{H}^3/\G.$ If $\mathscr{R}$ is incompressible then, subgroup $\pi_1(\mathscr{R})\subset\G$ have $\mathfrak{D}_{\pi_1(\mathscr{R})}=1$ 
which is contradiction. If $\mathscr{R}$ is compressible then, we can cut along compression disks. We either end with incompressible surface as before or after finitely many steps of cutting we obtain topological ball, which implies $\mathbb{H}^3/\G$ is handle-body, hence free. \par

Let $\G\in\mathfrak{J}^1_g$ be a non-classical Schottky group with $\mathfrak{D}_\G<1.$ Let $\{\G_n\}\subset\mathfrak{H}$, and $\G_n\to\G.$ Since $\G$ is Schottky group, it follows from Lemma \ref{open}, that there exists an open set $\mathscr{O}(\eta_\G)$
such that every element is quasi-circle of $\G$. Since $\G$ is a non-classical Schottky group, hence by Lemma \ref{singular}, we must have every open set about 
$\eta_\G$ contains some non-quasi-circle, in-particular we must have $\mathscr{O}(\eta_\G)$ contain a non-quasi-circle. Hence we must have $\G$ is a classical Schottky group. This implies $\mathfrak{J}_{g,o}\cap\mathfrak{J}_g^1$ is both open and closed in $\mathfrak{J}_g^1$ i.e. connected component of 
$\mathfrak{J}_g^1.$ Now it follows from Proposition \ref{H}, $\mathfrak{J}_{g,o}\cap\mathfrak{J}_g^1=\mathfrak{J}_g^1.$\par
Finally, sharpness comes from the fact that, there exists Kleinian groups which is not free of Hausdorff dimension equal to one. Hence we have our result.
\end{proof}
\noindent E-mail: geometricgroup@gmail.com\\
\pdfbookmark[1]{Reference}{Reference}
\bibliographystyle{plain}

\end{document}